\documentclass[11pt,a4paper]{amsart}
\usepackage{amsfonts}
\usepackage{amsthm}
\usepackage{amsmath}
\usepackage{amscd}
\usepackage[latin2]{inputenc}
\usepackage{t1enc}
\usepackage[mathscr]{eucal}
\usepackage{indentfirst}
\usepackage{graphicx}
\usepackage{graphics}
\usepackage{pict2e}
\usepackage{epic}
\numberwithin{equation}{section}
\usepackage[margin=2.5cm]{geometry}
\usepackage{epstopdf} 
\usepackage[table]{xcolor}
\usepackage{longtable}
\usepackage{changepage}
\usepackage{microtype}
\usepackage{float}
\usepackage[noadjust]{cite}
\usepackage{adjustbox}
\usepackage{array}
\usepackage{multirow}

\theoremstyle{plain}
\newtheorem{theorem}{Theorem}[section]

\theoremstyle{definition}
\newtheorem{definition}{Definition}[section]
\newtheorem{example}{Example}[section]
\newtheorem{remark} [theorem] {Remark}

\newcommand{\Q}{\mathbb{Q}}

\newcommand{\F}{\mathbb{F}}

\newcommand{\Proj}{\mathbb{P}}

\DeclareMathOperator{\Pf}{Pf}

\DeclareMathOperator{\wts}{wts}
\DeclareMathOperator{\Bl}{Bl}
\newcommand{\Span[1]}{\left< #1 \right>}

\usepackage{changepage}

\usepackage{makecell}
\usepackage[all]{xy}

\newcommand*\circled[1]{\tikz[baseline=(char.base)]{
		\node[shape=circle,draw,inner sep=1pt] (char) {#1};}}

\allowdisplaybreaks
\usepackage{mathpazo}
\usepackage{fancyhdr}
\usepackage{pgf,tikz}
\usetikzlibrary{arrows}
\usepackage{bm}
\usepackage{tikz}
\usepackage{lscape}
\usepackage{mathtools}

\usepackage[all]{xy}

\usepackage[colorlinks]{hyperref}

\usepackage{afterpage}
\usepackage{enumitem}
\usepackage{comment}
\usepackage{accents}
\usepackage{caption}
\usepackage{multicol}

\usepackage{indentfirst}
\usepackage{color}
\usepackage{graphicx}
\usepackage{newlfont}

\usepackage{amssymb}
\usepackage{amsmath}
\usepackage{latexsym}
\usepackage{amsthm}
\usepackage{mathrsfs}
\usepackage{amscd}

\usepackage[table]{xcolor}
\definecolor{lightgray}{gray}{0.9}

\begin{document}
	\makeatletter
	\@namedef{subjclassname@2020}{%
		\textup{2020} Mathematics Subject Classification}
	\makeatother
	
	\title{High-pliability Fano hypersurfaces}
	
	\author{Livia Campo}
	
	\address{School of Mathematics \\ 		
		Korea Institute for Advanced Study (KIAS) \\
		85 Hoegiro, Dongdaemun-gu \\
		Seoul, 02455 \\
		Republic of Korea}
	
	\email{liviacampo@kias.re.kr}

	
	
	\thanks{The author would like to thank Tiago Guerreiro for conversations and comments during the development of this work. The author is kindly supported by Korea Institute for Advanced Study, grant No. MG087901.}

	\begin{abstract} 
		We show that five of Reid's Fano 3-fold hyperurfaces containing at least one compound Du Val singularity of type $cA_n$ have pliability at least two. 
		The two elements of the pliability set are the singular hypersurface itself, and another non-isomorphic Fano hypersurface of the same degree, embedded in the same weighted projective space, but with different compound Du Val singularities. 
		The birational map between them is the composition of two birational links initiated by blowing up two Type I centres on a codimension 4 Fano 3-fold of $\Proj^2 \times \Proj^2$-type having Picard rank 2. 
	\end{abstract}
	
\maketitle

\section{Introduction}

We work over the complex numbers. 
In recent years, algebraic geometers have gained a new perspective on the birational classification of Fano varieties by looking at them through the lenses of the  Minimal Model Program (MMP)\cite{BCHM}. 
The strategy is to describe the birational classes of Fano varieties by studying "minimal" members of such families, called \textit{Mori fibre spaces}, that are the outcomes of the MMP. 
For 3-dimensional Fano varieties, these outcomes are not unique, albeit birational, so it makes sense to understand how they relate to one another using techniques coming from the Sarkisov Program \cite{Corti95,HaconMCKernanSarkisovProgram} and variation of Geometric Invariant Theory (vGIT) \cite{Corti2ray,BrownZucconi,CampoSarkisov}. 
The set of all the Mori fibre spaces that are birational to a given one has been formalised by Corti in the definition of \textit{pliability} (also in \cite[Foreword 4.6]{ExplicitBirGeom}).
\begin{definition}[Definition 1.2 of \cite{cortimella}]
	Given a Mori fibre space $X \rightarrow S$, the \textit{pliability} of $X$ is the set
	\begin{equation*}
	\mathcal{P}(X) \coloneqq \{ \text{Mfs } Y \rightarrow T \; | \; X \text{ is birational to } Y \} \, / \, \text{square equivalence}
	\end{equation*}
	where two Mori fibre spaces are \textit{square equivalent} if there are two birational maps $X \dashrightarrow Y$ and $S \dashrightarrow T$ such that the maps induced on the generic fibres $X_f \dashrightarrow Y_f$ are biregular. 
\end{definition}

In this paper we focus on $\Q$-factorial terminal 3-dimensional Fano varieties (3-folds) having Fano index 1. They can be embedded via Graded Rings methods into suitable weighted projective spaces (see \cite{AltinokBrownReidK3,ReidGradedRBirGeom}, and \cite{grdb} for the list of Hilbert series associated to terminal Fano 3-folds). When we talk about \textit{codimension} of a Fano 3-fold, we refer to its codimension with respect to its embedding into a weighted projective space $w\Proj^n$, for some $n >0$. 

Many efforts have been devoted to studying the pliability of Fano 3-folds. For instance, Reid's "famous 95" Fano hypersurfaces \cite{ReidCanonical3folds,IanoFletcher} have been proven to have pliability $|\mathcal{P}| =1$ both in the general quasi-smooth \cite{CPR} case and for all quasi-smooth members \cite{CheltsovParkRigidHypersurfaces}. 
Of the 85 Fano 3-fold codimension 2 complete intersections, 19 have pliability $|\mathcal{P}| =1$ \cite{IskovskikhPukhlikov,OkadaCI,AhmadinezhadZucconiCI}. In codimension 3, only 3 families have $|\mathcal{P}| =1$ \cite{AhmadinezhadOkadaPfaff}. 
These are the Fano 3-folds for which there is a structure theorem describing their equations (see also \cite{EisenbudBuchsbaum}). Starting from Fano 3-folds in codimension 4, a more sophisticated machinery is needed to retrieve their equations, called \textit{unprojection} \cite{KustinMiller,PapadakisComplexes,T&Jpart1}, that pinpoints between two and four distinguished deformation families (evocatively called Tom and Jerry) for each Hilbert series (see \ref{Setting and notation} below for details). 
These have different Euler characteristics; we call \textit{second Tom type} the Tom family with smaller Euler characteristic. 

Here we focus on those codimension 4 Fano 3-folds having Picard rank 2 as in \cite{BrownKasprzykQureshiP2xP2}, whose equations resemble the equations of the Segre embedding of $\Proj^2 \times \Proj^2$: these are realised as $2 \times 2$ minors of graded $3 \times 3$ matrices, and second Tom type Fano 3-folds always present a $\Proj^2 \times \Proj^2$-type structure. 
Such $\Proj^2 \times \Proj^2$-type Fano 3-folds are not Mori fibre spaces, but we can run birational links in the same spirit of the Sarkisov Program by performing a toric vGIT of a blow up of $w\Proj^n$ similarly to \cite{HamidPliabilityCox,BrownZucconi,CampoSarkisov}. 
We are most interested in the endpoints of these birational links: in fact, we restrict to those terminating with a Fano 3-fold hypersurface in a weighted $\Proj^4$. The possible hypersurfaces we obtain are listed in Table \ref{Endpoints table} below. 
We consider only codimension 4 Fano 3-folds of $\Proj^2 \times \Proj^2$-type and of second Tom type that allow for two distinct birational links both terminating with a Fano hypersurface. 
We compare the two hypersurfaces thus obtained; they are factorial (see Theorem \ref{factorial endpoints} below) and have terminal singularities, both cyclic quotient and compound Du Val \cite{YPGReid}. 

We use this approach to prove the following. 
\begin{theorem} \label{Pliability theorem}
	The 3-dimensional Fano hypersurfaces $\mathcal{N}\textdegree 1,2,3,4,5$ of Reid's 95 hypersurfaces having compound Du Val singularities as in Table \ref{Endpoints table} have pliability $|\mathcal{P}| \geq 2$.
\end{theorem}
The pliability of quartic Fano 3-folds having one compound Du Val singularity has been studied in \cite{cortimella}, where the authors prove that the pliability of such singular quartic 3-fold is exactly 2: it contains only another Mori fibre space, the quasi-smooth complete intersection of a cubic and a quartic in a weighted $\Proj^5$. 

In this paper we consider more singular Fano hypersurfaces whose singularities arise in the process of running a birational link starting from $X$ of $\Proj^2 \times \Proj^2$-type. 
In particular, we encounter a singular quartic Fano 3-fold too (i.e.~ Reid's $\mathcal{N}\textdegree 1$ and GRDB ID \#20521, see entry \#11125 of Table \ref{Endpoints table} below): however, our has three compound Du Val singularities instead of one as in \cite{cortimella} (one of which is the $cA_2$ singularity of \cite{cortimella}). 
We still retrieve the complete intersection $Y_{3,4}$ of \cite{cortimella} (see Table \cite[Entry \#11125]{BigTableLinks}), which is now not quasi-smooth but has compound Du Val singularities. 
In addition, we find that the pliability set of our singular quartic Fano also contains another, non-isomorphic, quartic Fano 3-fold having different compound Du Val singularities (cf also \cite{OkadaMFStructuresPartII,OkadaMFStructuresPartIII}). 

This is a recurrent phenomenon in the picture we study in this paper. 
In fact, the elements in the pliability sets of the singular Fano hypersurfaces we obtain are the hypersurface itself, and the same hypersurface with different compound Du Val singularities.

The main ingredient to prove Theorem \ref{Pliability theorem} is understanding the behaviour of the birational links starting from a codimension 4 Fano 3-fold $X$ of $\Proj^2 \times \Proj^2$-type, and how they are affected by $X$ having $\rho_X=2$. 
The following theorem summarises Theorems \ref{sim div contr} and \ref{cons div contr} below.
\begin{theorem}
	Let $X$ be a codimension 4 Fano 3-fold in second Tom format, and let $p \in X$ be a Type I centre. Suppose that the birational link centred at $p$ terminates with a Fano 3-fold hypersurface $X'$. Then, the link terminates with two divisorial contractions.
\end{theorem}
The fact that $\rho_X=2$ is reflected on the birational link of $(X,p)$; in other words, the birational link initiated by the Kawamata blow-up of a Type I centre $p$ on $X$ a codimension 4 Fano 3-fold having Picard rank $\rho_X=2$ detects $\rho_X$ by presenting two divisorial contractions. 
Hence, a birational link constructed in this way and run on a Fano 3-fold of unknow Picard rank $\rho$ can give a lower bound to $\rho$ itself. 
This is a remark that is used in \cite[Section 6]{CampoSarkisov} to deduce the Picard rank of other deformation families of codimension 4 Fano 3-folds.

In Section \ref{Setting P2xP2} we set the notation and we introduce the $\Proj^2 \times \Proj^2$ formats of \cite{BrownKasprzykQureshiP2xP2}. We then study the birational links starting from a codimension 4 Fano 3-fold of $\Proj^2 \times \Proj^2$-type in Section \ref{birational links}. 
Finally, in Section \ref{Very sing hypersurfaces} we analyse the hypersurfaces that are the endpoints of the birational links, and we discuss their generality under the unprojection assumptions. We conclude with a complete explicit example in Section \ref{Example}. 
Below is the table containing the detailed results for each family we considered.

\begin{adjustwidth}{-30em}{-4em}
	\begin{center} 
		\begin{longtable}{| c | c | c || c c | c |}
			\caption{Hypersurface endpoints} \label{Endpoints table}
			
			\centering
			\endfirsthead
			\endhead
			\hline
			ID of $X$ & Format & Link & ID of $X'$ & $X' = X_d \subset w\Proj^4$ & Singularities \\
			\hline \hline	
			\multirow{2}{2em}{\#4839} & T$_2$ $\bullet_{13,45}$ & $>>>$ & \#10980 & $X_8 \subset \Proj(1_{x_1},1_{x_2},1_{y_3},2_{y_2},4_{x_3})$ & $cA_5 \;@\; P_{y_3}$ \\ \cline{2-6}
			& T$_5$ $\bullet_{14}$ & $>>>$ & \#10980 & $X_8 \subset \Proj(1_{x_1},1_{x_2},1_{s},2_{y_4},4_{t})$ & $cA_6 \;@\; P_{s}$ \\
			\hline	\hline 	
			\multirow{2}{2em}{\#4915} & T$_2$ $\bullet_{13,45}$ & $>>>$ & \#10981 & $X_7 \subset \Proj(1_{x_1},1_{x_2},1_{y_3},2_{y_2},3_{x_3})$ & $cA_4 \;@\; P_{y_3}$ \\ \cline{2-6}
			& T$_5$ $\bullet_{14}$ & $>>>$ & \#10981 & $X_7 \subset \Proj(1_{x_1},1_{x_2},1_{s},2_{y_4},4_{t})$ & $cA_6 \;@\; P_{s}$ \\
			\hline \hline 
			\multirow{6}{2em}{\#5002} &  &  &  &  & $cA_5 \;@\; P_{y_2}$ \\ 
			& T$_2$ $\bullet_{13,45}$ & $>=>$ & \#16202 & $X_6 \subset \Proj(1_{x_1},1_{x_2},1_{y_2},1_{y_3},3_{x_3})$ & $cA_3 \;@\; P_{y_3}$ \\
			&&&&& $cA_5 \;@\; P \coloneqq \left( 0,0,0,-\frac{1}{2},1 \right)$ \\ \cline{2-6}
			&  &  &  &  & $cA_4 \;@\; P_{s}$ \\
			& T$_4$ $\bullet_{13}$ & $>=>$ & \#16202 & $X_6 \subset \Proj(1_{x_1},1_{x_2},1_{s},1_{y_4},3_{t})$ & $cA_5 \;@\; P_{y_4}$ \\
			&&&&& $cA_4 \;@\; P \coloneqq \left( 0,0,0,2,1 \right)$ \\
			\hline \hline 
			\multirow{2}{2em}{\#5163} & T$_2$ $\bullet_{13,45}$ & $>>>$ & \#11001 & $X_6 \subset \Proj(1_{x_1},1_{x_2},1_{y_3},2_{y_2},2_{x_3})$ & $cA_3 \;@\; P_{y_3}$ \\ \cline{2-6}
			& T$_5$ $\bullet_{14}$ & $>>>$ & \#11001 & $X_6 \subset \Proj(1_{x_1},1_{x_2},1_{s},2_{y_4},2_{t})$ & $cA_4 \;@\; P_{s}$ \\
			\hline	 \hline 
			\multirow{6}{2em}{\#5306} &  &  &  &  & $cA_3 \;@\; P_{y_2}$ \\ 
			& T$_2$ $\bullet_{13,45}$ & $>=>$ & \#16203 & $X_5 \subset \Proj(1_{x_1},1_{x_2},1_{y_2},1_{y_3},2_{x_3})$ & $cA_2 \;@\; P_{y_3}$ \\
			&&&&& $cA_3 \;@\; P \coloneqq \left( 0,0,0,1,-1 \right)$ \\ \cline{2-6}
			&  &  &  &  & $cA_4 \;@\; P_{s}$ \\
			& T$_4$ $\bullet_{13}$ & $>=>$ & \#16203 & $X_5 \subset \Proj(1_{x_1},1_{x_2},1_{s},1_{y_4},2_{t})$ & $cA_3 \;@\; P_{y_4}$ \\
			&&&&& $cA_3 \;@\; P \coloneqq \left( 0,0,0,1,1 \right)$ \\
			\hline \hline 
			\multirow{2}{2em}{\#10985} & T$_2$ $\bullet_{13,45}$ & $>>>$ & \#16203 & $X_5 \subset \Proj(1_{x_1},1_{x_2},1_{y_3},1_{x_3},2_{y_2})$ & $cA_2 \;@\; P_{y_3}$ \\ \cline{2-6}
			& T$_5$ $\bullet_{14}$ & $>>>$ & \#16203 & $X_5 \subset \Proj(1_{x_1},1_{x_2},1_{s},1_{y_4},2_{t})$ & $cA_3 \;@\; P_{s}$ \\
			\hline	\hline 	
			\multirow{6}{2em}{\#11125} &  &  &  &  & $cA_3 \;@\; P_{y_2}$ \\ 
			& T$_2$ $\bullet_{13,45}$ & $>=>$ & \#20521 & $X_4 \subset \Proj(1_{x_1},1_{x_2},1_{x_3},1_{y_2},1_{y_3})$ & $cA_2 \;@\; P_{y_3}$ \\
			&&&&& $cA_3 \;@\; P \coloneqq \left( 0,0,0,1,-1 \right)$ \\ \cline{2-6}
			&  &  &  &  & $cA_3 \;@\; P_{s}$ \\
			& T$_4$ $\bullet_{13}$ & $>=>$ & \#20521 & $X_4 \subset \Proj(1_{x_1},1_{x_2},1_{s},1_{y_4},1_{t})$ & $cA_2 \;@\; P_{y_4}$ \\
			&&&&& $cA_2 \;@\; P \coloneqq \left( 0,0,0,1,1 \right)$ \\
			\hline
		\end{longtable}
	\end{center}
\end{adjustwidth}
\paragraph{\textbf{Notation of Table \ref{Endpoints table}}}
The first and fourth columns, e.g.~ numbers preceded by \#, show the Graded Ring Database IDs of $X$ and the links' endpoints respectively. The second column records the format of $X$, e.g.~ T$_2$ $\bullet_{13,45}$ means that $X$ is in second Tom$_2$ format and the entries $m_{13}, m_{45} = 0$ (see Definition \ref{Tom definition} below). 
The third column contains information about the shape of the link, that is, the weights $d_i$ of $w\Proj^7$ are either $d_1 > d_2 > d_3 > d_4$ for $>>>$ or $d_1 > d_2 = d_3 > d_4$ for $>=>$. 
The fifth column shows the Fano hypersurface endpoint obtained via a link from $X$ in the given Tom format. 
Lastly, the sixth column reports the type of compound Du Val singularities of the endpoint, and at which point they are, e.g.~ the coordinate point $P_{y_3} \in X_8 \subset \Proj^4(1^3,2,4)$ is a $cA_5$ singularity.

\section{Fano varieties in $\Proj^2 \times \Proj^2$ Tom format} \label{Setting P2xP2}

\subsection{Setting and notation} \label{Setting and notation}

Codimension 4 Fano 3-folds obtained via Type I unprojections present between two and four formats (at most two Tom and Jerry formats respectively), identifying just as many distinguished deformation families (see \cite[Section 3.4]{T&Jpart1} and \cite{PapadakisReidKM,KustinMiller}). 
To build a codimension 4 Fano 3-fold $X$ we need the following data.
\begin{itemize}
	\item A fixed projective plane $D \coloneqq \Proj^2(a_{x_1},b_{x_2},c_{x_3}) \subset \Proj^6(a_{x_1},b_{x_2},c_{x_3}, d_{1_{y_1}}, d_{2_{y_2}}, d_{3_{y_3}} d_{4_{y_4}})$ with $d_1 \geq d_2 \geq d_3 \geq d_4$. So $D$ is defined by the ideal $I_D \coloneqq \Span[y_1,y_2,y_3,y_4]$.
	\item A family $\mathcal{Z}_1$ of codimension 3 Fano 3-folds $Z \subset w\Proj^6$, each defined by maximal pfaffians of a graded skew-symmetric $5 \times 5$ syzygy matrix $M$ with weighted entries 
	\begin{equation*}
	\left(
	\begin{array}{c c c c}
	m_{1,2} & m_{1,3} & m_{1,4} & m_{1,5} \\
	& m_{2,3} & m_{2,4} & m_{2,5} \\
	& & m_{3,4} & m_{3,5} \\
	& & & m_{4,5}
	\end{array}
	\right) 
	\end{equation*}
	where we omit the principal diagonal, and write only the upper-right triangle.
\end{itemize}
Each entry of $M$ is occupied by a polynomial in the degree dictated by the grading (all possible gradings are listed in \cite{TJBigTable}). 
The plane $D$ is a divisor inside $Z_1 \in \mathcal{Z}_1$ if the equations of $Z_1$ are the maximal pfaffians of $M$ in either Tom or Jerry format, and $Z_1$ is nodal on $D$. 
In this paper we will only focus on the deformation families arising as Tom formats.
\begin{definition}[\cite{T&Jpart1}, Definition 2.2]\label{Tom definition}
	A $5 \times 5$ skew-symmetric matrix $M$ is in Tom$_k$ format if and only if each entry $a_{i,j}$ for $i,j \neq k$ is in the ideal $I_D$.
\end{definition}
If $M$ is in Tom format, we say that
\begin{definition} \label{Tom type def}
	A codimension 4 index 1 Fano 3-fold $X$ is \textit{of Tom Type} if it is obtained as Type I unprojection of the codimension 3 pair $Z \coloneqq \left( \Pf_i(M) \right)_{i=1}^5 \supset D$ in a Tom family (cf \cite{T&Jpart1,PapadakisReidKM}). 
	It is said to be \textit{general} if $Z \supset D$ is general in its Tom family. 
	The image of $D \subset Z$ in $X$ is a cyclic quotient singularity $X \ni p \sim \frac{1}{r}(a,b,c)$ called \textit{Type I centre}. 
\end{definition}
Type I unprojections introduce a new coordinate $s$ of weight $r$, called \textit{unprojection variable}, and $X$ is now embedded into a 7-dimensional weighted projective space $\Proj^7(a,b,c, d_1, \dots, d_4,r)$. 
For $\Pf_i(M),\,g_j \in \mathbb{C}[x_1,x_2,x_3,y_1,\ldots,y_4]$, the equations of $X$ are of the form
\begin{equation*}
\left( \Pf_i(M) = sy_j-g_j=0, \,\, 1\leq i \leq 5,\, 1\leq j\leq 4 \right)
\end{equation*}
When there are two possible Tom formats for $M$, one of the two (the one with smaller Euler characteristic, which we will often call \textit{second Tom format} for simplicity) is such that some of the Tom constraints cannot be satisfied with the coordinates of $w\Proj^7$, e.g.~ a certain entry must be in $I_D$, but its degree is unattainable. 
As a consequence, some entries can only be filled with the zero polynomial. The bullet notation in Table \ref{Endpoints table} records the zero entries in $M$, e.g.~ for $M$ in T$_5$ $\bullet_{14}$ format, $m_{14}=0$. 

If $X$ is of second Tom type, it has Picard rank $\rho_X=2$ (see \cite[Proposition 2.1]{BrownKasprzykQureshiP2xP2}), and its equations can be written in the $\Proj^2 \times \Proj^2$ format of \cite{BrownKasprzykQureshiP2xP2}, that is they are the $2 \times 2$ minors of a $3 \times 3$ graded matrix 
\begin{equation*}
N \coloneqq \left(
\begin{array}{c c c}
n_{11} & n_{12} & n_{13} \\
n_{21} & n_{22} & n_{23} \\
n_{31} & n_{32} & n_{33}
\end{array}
\right) \; .
\end{equation*}
Building the matrix $M$ from $N$ consists in performing a Gorenstein projection with respect to $p \in X$ (see \cite[Subsection 3.2]{BrownKasprzykQureshiP2xP2} and the example \ref{Example} below).

\section{Birational links} \label{birational links}

We consider those Fano 3-folds $X$ in second Tom format that have at least two Type I centres whose blow up initiates a birational link terminating in a Fano hypersurface; we will then compare the Fano endpoints. 
Thus, this restricts the landscape to those Fano 3-folds $X$ sitting inside weighted $w\Proj^7$ having either $d_1 > d_2 > d_3 > d_4$ or $d_1 > d_2 = d_3 > d_4$ (see \cite[Lemma 4.8]{CampoSarkisov}). 
Each single birational link that we look at will therefore fall in one of these two cases: when $d_1 > d_2 > d_3 > d_4$ we have 
\begin{equation} \label{Birational link consecutive div}
\xymatrix@C=15pt{
	& Y_1 \ar[dl]_\Phi \ar[dr]^{\alpha_1} \ar@{-->}[rr]^{\Psi_1} & & Y_2 \ar[dl]_{\beta_1} \ar[dr]^{\alpha_2} \ar@{-->}[rr]^{\Psi_2} & & Y_3 \ar[dl]_{\beta_2} \ar[dr]^{ \Phi_1' \coloneqq \Psi_3 = \alpha_3} & &  \\
	w\Proj \supset X & & Z = Z_1 & & Z_2 & & Z_3 = Y_4 \ar[dr]^{\Phi_2'} & \\
	&&&&&&& X' \subset w\Proj'
}
\end{equation} 
and when $d_1 > d_2 = d_3 > d_4$ we have
\begin{equation} \label{Birational link simultaneous div}
\xymatrix@C=15pt{
	& Y_1 \ar[dl]_\Phi \ar[dr]^{\alpha_1} \ar@{-->}[rr]^{\Psi_1} & & Y_2 \ar[dl]_{\beta_1} \ar[dr]^{\alpha_2} \ar@{-->}[rr]^{\Psi_2} & & Y_3 \ar[dl]_{\beta_2} \ar[dr]^{\Phi'} &  \\
	w\Proj \supset X & & Z = Z_1 & & Z_2 & & X' \subset w\Proj'
}
\end{equation} 
where $\Phi'$ consists of two simultaneous divisorial contractions (see Theorem \ref{sim div contr}). 
Call $\F_i$ the rank two toric variety that is the ambient space of $Y_i$. 

We now list some theorems regarding the behaviour of the birational links starting from $(X,p)$. These ultimately justify the diagrams \eqref{Birational link consecutive div} and \eqref{Birational link simultaneous div}. 
Some of these results are proven in \cite{CampoSarkisov}, and they hold for any Tom format. 
Each link is initiated by a Kawamata blow-up of $X$ at its Type I centre $p$, and continues performing a variation of GIT quotient of the toric ambient space of $Y_1$: this is described explicitly in \cite[Section 3]{CampoSarkisov}. 
\begin{theorem}[Theorem 4.3 of \cite{CampoSarkisov}]
	Let $n$ be the number of nodes on $D \subset Z_1$. Then, the birational map $\Psi_1 \colon Y_1 \dashrightarrow Y_2$ is a flop of $n$ smooth rational curves.
\end{theorem}
\begin{theorem}[Theorem 4.5 of \cite{CampoSarkisov}] \label{hypersurf flips}
	Since $d_1 > d_2$, the birational map $\Psi_2 \colon Y_2 \dashrightarrow Y_3$ is a hypersurface flip.
\end{theorem}
\begin{proof}
	Refer to \cite[Theorem 4.5]{CampoSarkisov} for the proof of the fact that $\Psi_2$ is a flip. Here we show that, in the considered cases, $\Psi_2$ is in particular a hypersurface flip. 
	Note that the coordinate $y_1$ always occupies the entry $m_{35}, m_{34}, m_{25} \in I_D$ in the formats Tom$_2$, Tom$_5$, Tom$_4$ respectively, and also entry $m_{23} \not\in I_D$ in format Tom$_5$, and entries $m_{35} \in I_D$, $m_{24}, m_{34} \not\in I_D$ in format Tom$_4$. In the entries having degree $d_1$ occupied by $y_1$, we have $t y_1$ in the $\Phi$-pull-back. 
	For this reason, $t$ is locally eliminated in a neighbourhood of $P_{y_1} \in Z_2$ when $M$ is either in Tom$_4$ or Tom$_5$ format. 
	Instead, since $x_3$ occupies the entry $m_{12}$ in Tom$_2$ format (see also \eqref{>>>T2} below), $x_3$ is locally eliminated when $M$ is in Tom$_2$ format. 
	Observe that $y_4$ occupies entry $m_{14}$ in Tom$_2$, $m_{12}$ in Tom$_5$, and $m_{12}$ in Tom$_4$; hence, $y_4$ is locally eliminated too.
	Lastly, the unprojection variable $s$ is globally eliminated because the unprojection equations are of the form $s y_i = g_i (t,x_1,x_2,x_3,y_1,y_2,y_3,y_4)$. 
	
	This shows that the locus flipped by $\Psi_2$ is: the weighted plane $\Proj^2(d_1,a,b)$ cut by the remaining equations of $Y_2$ (the ones not used to eliminate variables); the weighted plane $\Proj^2(a,b,c)$ cut by the remaining equations of $Y_2$. 
	Thus, these are hypersurface flips.
\end{proof}

\begin{theorem} \label{cons div contr}
	Suppose that $d_1 > d_2 > d_3 > d_4$ and $X$ is in second Tom format. Then, both $\Phi_1' \coloneqq \Psi_3 \colon Y_3 \dashrightarrow Y_4$ and $\Phi_2' \colon Y_4 \dashrightarrow X'$ are divisorial contractions to a point.
\end{theorem}
\begin{proof}
	Suppose that $d_1 > d_2 > d_3 > d_4$. In this instance and for the families we consider, the matrix $M$ is either in second Tom$_2$ format \eqref{>>>T2} or is second Tom$_5$ format \eqref{>>>T5}, having grading of the form 
	
	\noindent\begin{minipage}{.5\linewidth}
	\begin{equation} \label{>>>T2}
	\left(
	\begin{array}{c c c c}
	\textcolor{purple}{c} & c + d_2 - d_3 & d_4 & d_3 \\
	& \textcolor{purple}{d_4} & \textcolor{purple}{d_3} & \textcolor{purple}{d_2} \\
	& & d_2 & d_1 \\
	& & & d_1 + d_3 - d_4
	\end{array}
	\right) \; \text{ or}
	\end{equation}
	\end{minipage}%
	\begin{minipage}{.5\linewidth}
	\begin{equation} \label{>>>T5}
	\left(
	\begin{array}{c c c c}
	d_4 & d_3 & d_3 & \textcolor{purple}{d_2} \\
	& d_2 & d_2 & \textcolor{purple}{d_1} \\
	& & d_1 & \textcolor{purple}{c} \\
	& & & \textcolor{purple}{c}
	\end{array}
	\right) 
	\end{equation}
	\end{minipage}
	where the entries with degree written in purple are polynomials not in the ideal $I_D$. 
	Then, the second Pfaffian $\Pf_2(M)$ for $M$ as in \eqref{>>>T2}, or $\Pf_5(M)$ for $M$ as in \eqref{>>>T5}, contains the pure monomial $y_2 y_3$. Thus, the localisation at $P_{y_2} \in Z_3$ is such that the coordinate $y_3$ can be locally eliminated in a neighbourhood of $P_{y_2}$. 
	The unprojection variable $s$ is always globally eliminated thanks to the unprojection equations. 
	Moreover, the entry $m_{25}$ in \eqref{>>>T2} and the entry $m_{15}$ in \eqref{>>>T5} are polynomials in which $y_2$ appears as a pure monomial of degree 1: in the pull-back by $\Phi$ it will become $t y_2$ (see also \cite[Subsection 3.3]{CampoSarkisov} for more details). Thus, the localisation $P_{y_2} \in Z_3$ locally eliminates the variable $t$. 
	
	The $m_{12}$ entry of \eqref{>>>T2}, and the $m_{35}$ entry of \eqref{>>>T2} must feature $x_3$. Therefore, the localisation at $P_{y_2}$ also locally eliminates $x_3$ via $\Pf_5(M)$ in \eqref{>>>T2} and via $\Pf_1(M)$ in \eqref{>>>T5} (note that $m_{14}=0$).
	
	Therefore, the contracted locus of the toric flip $\Psi_3 \colon \F_3 \dashrightarrow \F_4$ restricted to $Y_3$ is the weighted plane $\Proj^2(a,b,d_1-d_2)$; this is contracted to $P_{y_2} \in Z_3$, and we call $\Phi_1'$ the restriction $\Psi_3|_{Z_3}$. 
	
	The singularities of $Y_3$ are terminal: they are either isolated cyclic quotient singularities, or isolated compound Du Val singularities. Also, $Y_3$ is $\Q$-factorial and its anticanonical divisor is $\Phi_1'$-ample. 
	The contracted locus is irreducible, and $\Phi_1'$ has relative Picard rank 1. Thus, $\Phi_1'$ is a divisorial contraction to $P_{y_2} \in Z_3$. 
	Proving that $\Phi_2'$ is a divisorial contraction is analogous. 
\end{proof}

\begin{theorem} \label{sim div contr}
	Suppose that $d_1 > d_2 = d_3 > d_4$ and $X$ is in second Tom format. Then, $\Phi' \colon Y_3 \dashrightarrow X'$ consists of two simultaneous divisorial contractions to points.
\end{theorem}
\begin{proof}
	When  $d_1 > d_2 = d_3 > d_4$ and for the families we consider, $M$ is either in second Tom$_2$ format or in second Tom$_4$ format, whose gradings are as follows
	
	\noindent\begin{minipage}{.5\linewidth}
		\begin{equation} \label{>=>T2}
		\left(
		\begin{array}{c c c c}
		\textcolor{purple}{c} & c & d_4 & d_2 \\
		& \textcolor{purple}{d_4} & \textcolor{purple}{d_2} & \textcolor{purple}{d_1} \\
		& & d_2 & d_1 \\
		& & & d_1 + d_4 - c
		\end{array}
		\right) \; \text{ or}
		\end{equation}
	\end{minipage}%
	\begin{minipage}{.5\linewidth}
		\begin{equation} \label{>=>T4}
		\left(
		\begin{array}{c c c c}
		d_4 & d_4 & \textcolor{purple}{d_2} & d_2 \\
		& d_2 & \textcolor{purple}{d_1} & d_1 \\
		& & \textcolor{purple}{d_1} & d_1 \\
		& & & \textcolor{purple}{c}
		\end{array}
		\right) 
		\end{equation}
	\end{minipage}
	where the entries with purple degree are not in the ideal $I_D$. 
	
	Torically, $\Phi'$ is a divisorial contraction to the curve $\Proj^1_{y_2:y_3} \subset w\Proj'$. However, the generic point $p \coloneqq (\sigma_1, \sigma_2) \in \Proj^1_{y_2:y_3}$ for $\sigma_1, \sigma_2 \not= 0$ is not contained in $X'$. 
	Indeed, both $\Pf_2(M)$ and $\Pf_4(M)$ in \eqref{>=>T2}, \eqref{>=>T4} respectively are equal to $y_2 y_3 - y_1 y_4$, which does not vanish at $p$. 
	Therefore, the toric contraction restrict to two simultaneous divisorial contractions, one at $P_{y_2} \in X'$ and one at $P_{y_3} \in X'$.
\end{proof}
Note that Theorems \ref{cons div contr}, \ref{sim div contr} together with the fact that $\rho_X = 2$ imply that the Picard rank of the endpoint hypersurfaces is 1.

\subsection{Birational links through $\Proj^2 \times \Proj^2$ formats}

Suppose a codimension 4 Fano 3-fold $X$ has two different Type I centres $p,q$, and that the birational links initiated by blowing up $p,q$ end with two different hypersurfaces $X'$ and $X''$. 
For instance, if the weights of $w\Proj^7$ are $d_1 > d_2 > d_3 > d_4$ for both centres, we have diagrams like
\begin{equation} \label{two links diagram}
\xymatrix@C=15pt{
	&& W_3 \ar[dl]_{ \Theta_1'} & W_2 \ar@{-->}[l]_{\Xi_2} & W_1 \ar[dr]_{\Theta=\Bl_q} \ar@{-->}[l]_{\Xi_1} && Y_1 \ar[dl]^{\Phi=\Bl_p} \ar@{-->}[r]^{\Psi_1} & Y_2 \ar@{-->}[r]^{\Psi_2} & Y_3 \ar[dr]^{ \Phi_1'} &  \\
	& W_4 \ar[dl]_{\Theta_2'} &&& Z' \ar@{-->}[r] & X & Z \ar@{-->}[l] &&& Y_4 \ar[dr]^{\Phi_2'} & \\
	X'' &&&&&&&&&& X'
}
\end{equation}
where $Z \dashrightarrow X$ is the unprojection from $(Z,D)$ producing $X$ with Type I centre $p$, and, similarly, $Z' \dashrightarrow X$ unprojects $(Z',D')$ with Type I centre $q$. Note that $Z$ and $Z'$ are the bases for the flops of $\Psi_1$ and $\Xi_1$ respectively.

Our goal is to compare the hypersurface endpoints $X', X''$. We prove the following theorem. 
\begin{theorem} \label{endpoints not isomorphic}
	Let $X$ be a codimension 4 Fano 3-fold of second Tom type with at least two Type I centres $p,q \in X$ as in Table \ref{Endpoints table}. The endpoints $X', X''$ of the two birational links centred at $p$ and $q$ are non-isomorphic 3-dimensional Fano hypersurfaces having the same Hilbert series.
\end{theorem}
The proof of Theorem \ref{endpoints not isomorphic} is contained in Section \ref{Very sing hypersurfaces}, as it relies on the analysis of the compound Du Val singularities of $X'$ and $X''$ carried out in the proof of Theorem \ref{constant cDV} below.

The heart of the construction is to realise $(X,p)$ via Type I unprojection of $(Z,D)$ where $Z = \left( \Pf_i(M)=0 \right)_{i=1}^5$, and to then retrieve the defining matrix $M'$ of $(Z', D')$ by passing through the $\Proj^2 \times \Proj^2$ construction of $N$. 
Then, using the new pair $(Z',D')$ we perform the Type I unprojection to obtain $(X,q)$. Note that the equations of $(X,p)$ and $(X,q)$ are different, but they are two general members of the same Tom deformation family. 

\begin{example} \label{4839 P2xP2}
For instance, looking at Table \ref{Endpoints table} and in comparison with \cite{TJBigTable,BigTableLinks}, consider \#4839; the two Tom formats Tom$_2 \bullet_{13,45}$ and Tom$_5 \bullet_{14}$ represent the same deformation family with Hilbert series \#4839, but with respect to the two centres $p \sim \frac{1}{5}(1,1,4)$ and $q \sim \frac{1}{9}(1,1,8)$. 
Thus, we have $M$ in Tom$_2 \bullet_{13,45}$, $M'$ in Tom$_5 \bullet_{14}$, and $N$ as in \cite[Table 2]{BrownKasprzykQureshiP2xP2}. 

Consider the Type I centre $p \in X \subset \Proj^7(1_{x_1}, 1_{x_2}, 4_{x_3}, 5_s, 6_{y_4}, 7_{y_3}, 8_{y_2}, 9_{y_1})$ at the coordinate point $P_s$ having local coordinates $x_1,x_2,x_3$. 
The unprojection that produced such Type I centre started from $(Z,D)$ for $D = \Proj^2(1_{x_1}, 1_{x_2}, 4_{x_3})$ inside $Z = \left( \Pf_i(M)=0 \right)_{i=1}^5$ where $I_D \coloneqq \Span[y_1,y_2,y_3,y_4]$ and the grading of $M$ is below \eqref{matrices diagram}. 
The circled weights in \eqref{matrices diagram} mark the zero entries, and in purple are the entries that must be filled with polynomials not in $I_D$ as before. 
The matrix $N$ is built by ignoring the zero entries of $M$ and by placing the unprojection variable $s$ in the entry $n_{31}$ of $N$ (in a square box below). 

Call $N'$ the $3 \times 3$ matrix where the second and third row of $N$ have been swapped. The $n'_{33} = n_{23}$ entry of $N'$ (boxed) is therefore occupied by $y_1$, which is the unprojection variable when the Type I centre is $q \sim \frac{1}{9}(1,1,8)$ having local coordinates $x_1,x_2,y_2$. Projecting away from $q$ we retrieve $D' = \Proj^2(1_{x_1}, 1_{x_2}, 8_{y_2})$ sitting inside $Z' = \left( \Pf_i(M')=0 \right)_{i=1}^5$ where $I_{D'} \coloneqq \Span[x_3,s,y_3,y_4]$ and $M'$ is in Tom$_5 \bullet_{14}$ format (entries not in $I_{D'}$ marked in blue). 
\begin{equation} \label{matrices diagram}
\xymatrix@C=15pt{
	& {N \colon \left(
	\begin{array}{c c c}
	4 & 6 & 7 \\
	6 & 8 & 9 \\
	\boxed{5} & 7 & 8
	\end{array}
	\right)} \ar@{-->}[dl]_{\text{Gorenstein $p$-projection}}^{p \sim \frac{1}{5}(1,1,4)} \ar@{~>}[r] & {N' \colon \left(
	\begin{array}{c c c}
	4 & 6 & 7 \\
	5 & 7 & 8  \\
	6 & 8 & \boxed{9}
	\end{array}
	\right)} \ar@{-->}[dr]^{\text{Gorenstein $q$-projection}}_{q \sim \frac{1}{9}(1,1,8)} & \\
	{M \colon \left(\begin{array}{c c c c}
		\textcolor{purple}{4} & \circled{5} & 6 & 7 \\
		& \textcolor{purple}{6} & \textcolor{purple}{7} & \textcolor{purple}{8} \\
		& & 8 & 9 \\
		& & & \circled{10}
		\end{array}\right)} & & & {M' \colon \left(\begin{array}{c c c c}
		4 & 5 & \circled{5} & \textcolor{blue}{6} \\
		& 6 & 6 & \textcolor{blue}{7} \\
		& & 7 & \textcolor{blue}{7 }\\
		& & & \textcolor{blue}{8}
		\end{array}\right)}
}
\end{equation}
The codimension 4 Fano 3-fold $X$ realised as Type I unprojection of $(Z,D)$ lies in the same deformation family of the one realised as Type I unprojection of $(Z',D')$. 
We run the two birational links in \eqref{two links diagram} by blowing up $p$ and $q$, and using the two sets of equations for $X$ that we have just built using $\Proj^2 \times \Proj^2$ formats. We will give a detailed account on how to do so in the continuation of this explicit example in Section \ref{Example}.
\end{example}

\section{Very singular Fano hypersurfaces} \label{Very sing hypersurfaces}

The endpoints of the birational links we have constructed are Fano hypersurfaces in weighted $w\Proj^4$. Here we want to study their singularities, and discuss the generality of these hypersurfaces. 
Recall that, by Theorem \ref{hypersurf flips}, the birational links in this paper all present a hypersurface flip. This introduces a compound Du Val singularity on $Y_3$ induced by the local equation in the flipped locus of $\Psi_2$. 
Consequently, $X'$ has at least one compound Du Val singularity (possibly more, depending on whether the contracted locus of the divisorial contractions has a local equation too). 
\begin{theorem} \label{constant cDV}
	Let $X$ be a codimension 4 Fano 3-fold of second Tom type, and let $X$ be general in its Tom format. Assume that the birational link initiated by the Kawamata blowup of one of $X$'s Type I centres terminates with a Fano hypersurface $X'$. 
	Then, the compound Du Val singularities of $X'$ are constant when $X$ varies in its general Tom family. 
\end{theorem}
\begin{proof}
	Suppose the Type I centres are $p,q \in X$ as before. 
	For simplicity, we focus on the link initiated by blowing up $p$; the proof for the link corresponding to blowing up $q$ is analogous.
	Since we consider only weights $d_1 > d_2 > d_3 > d_4$ and $d_1 > d_2 = d_3 > d_4$, the second Tom formats we have are Tom$_2$, Tom$_5$, or Tom$_4$. Firstly, consider the case in which $d_1 > d_2 > d_3 > d_4$. 
	To fix ideas, suppose that $p \in X$ is the cyclic quotient singularity with lowest index in the basket of $X$. So, by \cite{TJBigTable}, $M$ is in Tom$_2 \, \bullet_{13,45}$ format as follows
	\begin{equation*}
	\left(
	\begin{array}{c c c c}
	x_3 & 0 & y_4 & y_3 \\
	& p_1 & p_2 & p_3 \\
	& & y_2 & y_1 \\
	& & & 0
	\end{array}
	\right) 
	\end{equation*}
	for $p_1,p_2,p_3 \not \in I_D$ homogeneous polynomials of prescribed degree as in \eqref{>>>T2}. 
	
	The equations of $X'$ can be found by imposing $y_4 = 1$ (e.g.~blowing down) in the equations of $Y_4$ (or $Y_3$ when $Y_4 = Y_3$); some coordinates (like $s$ via the unprojection equations) can be globally eliminated, and can be substituted in the rest of the equations of $Y_4 |_{y_4=1}$. 
	After the elimination of $s$, we are left with only the five pfaffian equations. In the cases we consider, the coordinates $t$ and $y_1$ can always be globally eliminated by $\Pf_5(M)$ and $\Pf_2(M)$ respectively for $M$ is in Tom$_2$ format. 
	The fifth pfaffian equation is $\Pf_5(M) = x_3 y_2 - 0 \cdot p_2 + y_4 p_1$. Here $p_1 \not \in I_D$ is a homogeneous polynomial of degree $d_4$ of the form $p_1 = y_4 + x_1^{d_4} + h_1$, where $h_1$ is a homogeneous polynomial of degree $d_4$, possibly in $I_D$ (also recall that all these second Tom type Fano 3-folds have Fano index 1, so we can assume that the weight of $x_1$ is 1). 
	The only monomials in $h_1$ that contribute to achieving quasi-smoothness of $X$ and reducedness of $D \subset Z$ are the pure monomials in $x_1,x_2,x_3$, and quasi-linear terms in $x_1,x_2,x_3$ (see \cite{CampoSarkisov}); indeed, monomials that are in terms of three or more variables vanish when in the restriction to $D$ when checking reducedness and quasi-smoothness. 
	Thus, we can assume without loss of generality that $h_1$ does not contain other monomials in $I_D$. As a consequence, in the pull-back of the pfaffian equations via $\Phi$, the only term in $\Pf_5(M)$ picking up a $t$ factor is $y_4$ appearing in $p_1$. Hence, the elimination of $t$ is global. 
	If we look at the second pfaffian instead, we have that $\Pf_2(M) = y_2 y_3 - y_1 y_4$; here the global elimination is evident.
	
	This leaves us with three pfaffian equations: $\Pf_1(M)$, $\Pf_3(M)$, and $\Pf_4(M)$. After replacing the global expressions of $t, y_1$, and since the equations are explicit, we can see that $\Pf_4(M)$ is always identically 0, and that $\Pf_1(M)$ is a $y_2$-multiple of $\Pf_3(M)$. Thus, the ideal of $\Pf_1(M), \Pf_3(M), \Pf_4(M)$ is actually only generated by $\Pf_3(M)$; this is the equation of the Fano (cf \cite[Lemma 4.8]{CampoSarkisov}) hypersurface $X'$. 
	Since $\Pf_3(M) = 0$ is an explicit equation, we can run a simple \verb|Magma| routine to find out where the compound Du Val singularities are on $X'$: it is enough to find the singular locus and its prime components. 
	For $d_1 > d_2 > d_3 > d_4$ and $M$ in Tom$_2$ format, then there is one compound Du Val singularity on $X'$ at the coordinate point $P_{y_3} \in X' \subset w\Proj^4$. 
	In order to understand what type of compound Du Val singularity there is at $P_{y_3}$, we look at $\Pf_3(M) = 0$ after replacing $t$ and $y_1$ with their local expressions. So, to $\Pf_3(M) = y_3 p_2(t,y_3,x_1,x_2,x_3) - y_4 p_3(t,y_2,x_1,x_2,x_3)$ we need to impose: $y_4=1$ to blow down; $y_3=1$ to localise at the singular point; $t = x_1^{d_4} + h_1$ for the global elimination of $t$. 
	Since $p_2 = y_3 + x_1^{d_3} + h_2$, the pullback via $\Phi$ and the localisation at $P_{y_3}$ imply that $\Pf_3(M) = t + x_1^{d_3} + h_2(t,y_3,x_1,x_2,x_3) - p_3(t,y_2,x_1,x_2,x_3)$. Thus, recalling the form of $\Pf_5(M), h_1$ the equation of $X'$ contains the monomials: $x_3 y_2$, $x_1^{d_4}$ (taking the smallest pure power of $x_1$), and $x_3^2$ (that appears in either $p_2$ or $p_3$ if $d_3,d_4 = 2 \cdot \wts(x_3)$). Note that, when $d_3,d_4 = 2 \cdot \wts(x_3)$ does not happen, we can still perform a local change of variables in a neighbourhood of $P_{y_3}$ to get two quadratic terms starting from $x_3 y_2$. 
	Locally at $P_{y_3}$ the equation of $X'$ is of the form $x^2 + x y + z^n$ plus higher order terms, which, by a change of variables, becomes the standard $x^2 + y^2 + z^n$ that expresses a compound Du Val singularity of type $cA_{n-1}$.

	The case of $d_1 > d_2 = d_3 > d_4$ is similar: the local expressions of $t, y_1$ are analogous to the case above, and $\Pf_3(M)$ is still the equation of $X'$. What changes is that the compound Du Val singularities are three: two at the coordinate points $P_{y_2}, P_{y_3}$, and one at a point $P$ indicated explicitly for each format in the right-most column of Table \ref{Endpoints table}. The singularity analysis is conducted separately at each point in a similar fashion as above (with a simple change of variables to make $P$ a coordinate point).

	In conclusion, the only constraints that we imposed to the generality of these $cA_n$ singularities are the following: the Tom constraints, the reducedness of $D \subset Z$, the quasi-smoothness of $X$. Thus, we conclude that the $cA_n$ singularities on $X'$ persist when $X$ varies in its Tom family. 
\end{proof}

\begin{remark}
	Note that there is possibly another compound Du Val singularity at orbifold points when $d_1 > d_2 > d_3 > d_4$ (more precisely, at the coordinate point $P_{y_2}$). Even though we do not study their singularity type here, their existence does not affect the statement of Theorem \ref{endpoints not isomorphic}. However, this could lead to yet another Mori fibre space birational to $X'$. Thus, the inequality in Theorem \ref{Pliability theorem} about the size of the pliability set of $X'$ could be strict. 
\end{remark}

\begin{proof}[Proof of Theorem \ref{endpoints not isomorphic}] 	
	The link initiated by $q \in X$ starts from $X$ in either Tom$_5$ or Tom$_4$ type. For simplicity we assume that $M'$ is in Tom$_5$ format; the proof for $M'$ is in Tom$_4$ format is analogous. 
	
	Recall from the proof of Theorem \ref{constant cDV} that the coordinate point $P_{y_3} \in X'$ has local equation $x^2 + y^2 + z^n$ where $n=d_4$; so, $P_{y_3}$ is a $cA_{d_4-1}$ singularity. Moreover, since the coordinates $t,s,y_1$ can be globally eliminated and $y_4$ has been set to 1, $X'$ sits inside the weighed projective space $\Proj^4(1_{x_1}, b_{x_2}, c_{x_3}, (d_2-d_4)_{y_2}, (d_3-d_4)_{y_3})$ (with a little abuse of notation we still use the $x_i,y_j$ coordinate notation here). 
	The proof of Theorem \ref{constant cDV} can be conducted similarly when considering the birational link initiated by blowing up the other Type I centre $q \in X$. 
	In this case, the globally eliminated coordinates are $y_1, y_2, y_3$ (where $y_1$ is the new unprojection variable coming from the Type I centre $q$), so the Fano hypersurface $X''$ is embedded in the weighed projective space $\Proj^4(1_{x_1}, b_{x_2}, c_{t}, (r-c)_{s}, (d_4-c)_{y_4})$. 
	As it was for $P_{y_3} \in X'$, the coordinate point $P_s \in X''$ is always a compound Du Val singularity with local equation of the form $x^2 + y^2 + z^{n'}$. 
	The equation of $X''$ is given by $\Pf_4(M')|_{x_3=1} = p_3 - s p_2 + y_4 p_3$, which contains the monomial $x_1^{d_2}$ coming from $p_2$. Note that $x_1^{d_2}$ is the smallest pure power of $x_1$ appearing in $\Pf_4(M')|_{x_3=1}$, because $p_3$ features $x_1^{d_1}$, and $d_1 > d_2$. 
	Thus, $n' = d_2$ and $P_s$ is a compound Du Val singularity of Type $cA_{d_2-1}$.  
	
	Note that the weighted projective spaces where  $X'$ and $X''$ are embedded have the same weights: indeed, in the cases we consider, $d_3-d_4 = r-c = 1$, and $d_2-d_4 = d_4 - c$.  
	
	The fourth pfaffian of $M'$ is $\Pf_4(M') = x_3 p_3 - s p_2 + y_4 p_1$, which is a homogeneous polynomial of degree $c + d_2$. The equation of $X''$ is given by $\Pf_4(M')_{x_3=1}$, which has therefore degree $d_2$ and is homogeneous in $\Proj^4(1_{x_1}, b_{x_2}, c_{t}, (r-c)_{s}, (d_4-c)_{y_4})$. 
	
	On the other hand, the third pfaffian of $M$ is $\Pf_3(M) = y_3 p_2 - y_4 p_3$, which is homogeneous of degree $d_4 + d_2$. The equation of $X'$ is $\Pf_3(M)_{y_4=1}$, whose degree is then $d_2$, and is homogeneous in  $\Proj^4(1_{x_1}, b_{x_2}, c_{x_3}, (d_2-d_4)_{y_2}, (d_3-d_4)_{y_3})$. 
	
	Thus, $X'$ and $X''$ are two Fano hypersurfaces having the same degree and sitting inside the same weighted $\Proj^4$. As a consequence, they have the same Hilbert series as in \cite{grdb}. However, they are not isomorphic because they have different compound Du Val singularities. 
\end{proof}

Now we have all the elements to prove our Main Theorem \ref{Pliability theorem}.
\begin{proof}[Proof of Theorem \ref{Pliability theorem}]
	In the cases we consider, the square-equivalence classes in the pliability definition reduce to isomorphism classes because $S = T = \{ \text{point} \}$. 
	Theorem \ref{endpoints not isomorphic} shows that the two birational links starting from $(X,p)$ and $(X,q)$ respectively terminate with two birational non-isomorphic Fano hypersurfaces $X', X''$ having the same Hilbert series but different compound Du Val singularities. 
	Hence, the pliability of $X'$ contains at least $X''$, so $|\mathcal{P}(X')| \geq 2$. 
\end{proof}

Lastly, we discuss the factoriality of the hypersurfaces we obtained. Factoriality is crucial for understanding how many compound Du Val singularities (including ordinary double points) a hypersurface can have in order to still have pliability $|\mathcal{P}| =1$. 
This is in relation to Conjecture 1.4 of \cite[extended version]{cortimella} and \cite{CheltsovFactorialityNodal3folds,CheltsovFactorial3foldsHypersurf,CheltsovParkFactorialHypersurfNodes, AbbanKaloghiros}.
Factoriality is necessary, but not sufficient, to prove the birational rigidity (i.e.~$|\mathcal{P}| =1$) of a hypersurface; in fact, the singular quartic 3-fold in \cite{cortimella} is factorial but has $|\mathcal{P}| =2$. 
The trend is that factoriality combined with a bounded number of ordinary double points gives birational rigidity.
What we show is that all our very singular Fano hypersurfaces are factorial and yet are non-rigid. 
In our case, the factoriality property still derives from the second Tom formats of $M, M'$, and therefore from the shape of the equations of $Z,Z'$. We check this in the next theorem. 
\begin{theorem} \label{factorial endpoints}
	The Fano 3-fold hypersurfaces $X'$ and $X''$ obtained as above are factorial. 
\end{theorem}

\begin{proof}
	We refer to the proofs of Theorems \ref{endpoints not isomorphic} and \ref{constant cDV} for the notation and the equations of $X'$ and $X''$. 
	For simplicity, we focus on $X'$; the proof for the hypersurface endpoint $X''$ is analogous. 
	To prove factoriality of $X'$ it is enough to show that they do not contain reducible curves such as $(x=y=0) \cap X'$ for some coordinates $x,y$ of $w\Proj^4$. 
	We look at the equation of $X'$, which is given by $\Pf_3(M)=0$ after evaluating at $y_4=1$ and replacing the expressions of the global elimination of $t, y_1$. 
	Recall that the polynomial $p_3$ contains $x_1^{d_2}$ without loss of generality. Therefore, whenever we cut $X'$ by $(x=y=0)$ with both $x,y \not= x_1$, $X'$ does not contain $(x=y=0)$. 
	Suppose instead that $x=x_1$; we have two possibilities. 
	If $y=x_2$ or $x_3$, since the global expression of $t$ comes from $\Pf_5(M)$, we have that at least one of the monomials $y_3^2 x_2^m, y_3^2 x_2^{m'}$ appear in $\Pf_3(M)$ for some $m, m'$. Thus, $X'$ does not contain $(x=y=0)$. 
	On the other hand, if $y=y_i$ for $i \in \{ 1, \dots, 4\}$, then $p_3$ contains at least one of $x_2^{m}, x_3^{m'}$ for $m \coloneqq d_2 - \wts(x_2)$ and $m' \coloneqq d_2 - \wts(x_3)$; so, $X'$ does not contain $(x=y=0)$. 
\end{proof}

\section{Example} \label{Example}

We run the explicit computation of the matrices $M, M', N, N'$ for the family $X$ with Hilbert series \#4839 expanding from Example \ref{4839 P2xP2} above. 
Moreover, we study the two links obtained by blowing up $X$ at its two Type I centres $p \sim \frac{1}{5}(1,1,4)$ and $q \sim \frac{1}{9}(1,1,8)$; this produces a diagram as in \eqref{two links diagram}. Then, we find the explicit equations of the endpoints $X,X''$ and we look at their compound Du Val singularities, following the proof of Theorem \ref{constant cDV}. 

Continuing Example \ref{4839 P2xP2} above, the $5 \times 5$ skew-symmetric matrix $M$ is 
\begin{equation*} 
\xymatrix@C=15pt{
	{M \coloneqq \left(\begin{array}{c c c c}
	\textcolor{purple}{x_3} & 0 & y_4 & y_3 \\
	& \textcolor{purple}{p_1} & \textcolor{purple}{p_2} & \textcolor{purple}{p_3} \\
	& & y_2 & y_1 \\
	& & & 0
	\end{array}\right)} & {\text{with grading}} & 
	{\left(\begin{array}{c c c c}
		\textcolor{purple}{4} & \circled{5} & 6 & 7 \\
		& \textcolor{purple}{6} & \textcolor{purple}{7} & \textcolor{purple}{8} \\
		& & 8 & 9 \\
		& & & \circled{10}
		\end{array}\right)} 
}
\end{equation*}
where $p_{5-i} = x_1^{d_i} + x_2^{d_i} + y_i + h_{5-i}(x_1,x_2,x_3)$ and $h_{5-i}$ is a homogeneous polynomial of degree $d_{5-i}$, $i=2,3,4$. 
In fact, we only need at least one of the $x_i$ to appear as a pure power in the polynomials $p_j$. 
The $\Proj^2 \times \Proj^2$ format $N$ for $(X,p)$ is therefore
\begin{equation*} 
\xymatrix@C=15pt{
	{N \coloneqq \left(
		\begin{array}{c c c}
		x_3 & y_4 & y_3 \\
		p_1 & y_2 & y_1 \\
		s & p_2 & p_3
		\end{array}
		\right)} & {\text{with grading}} & 
	{\left(
		\begin{array}{c c c}
		4 & 6 & 7 \\
		6 & 8 & 9 \\
		\boxed{5} & 7 & 8
		\end{array}
		\right)} 
} \; .
\end{equation*}
The other $\Proj^2 \times \Proj^2$ format $N'$, corresponding to $(X,q)$, is the matrix obtained by swapping the second and third row of $N$. With a Gorenstein projection from $q \in X$ we get
\begin{equation*} 
\xymatrix@C=15pt{
	{M' \coloneqq \left(\begin{array}{c c c c}
		x_3 & s & 0 & \textcolor{blue}{p_1} \\
		& y_4 & y_4 & \textcolor{blue}{p_2} \\
		& & y_3 & \textcolor{blue}{p_3}\\
		& & & \textcolor{blue}{y_2}
		\end{array}\right)} & {\text{with grading}} & 
	{\left(\begin{array}{c c c c}
		4 & 5 & \circled{5} & \textcolor{blue}{6} \\
		& 6 & 6 & \textcolor{blue}{7} \\
		& & 7 & \textcolor{blue}{7 }\\
		& & & \textcolor{blue}{8}
		\end{array}\right)}
} \; .
\end{equation*}
To perform the blow-ups of $X$ at $p$ and $q$, we follow \cite[Section 3.4]{CampoSarkisov}; thus, the proper transforms (after saturation by a $t$ factor) $Y_1$ and $W_1$ of $X$ via such blow-ups are respectively embedded into the rank 2 toric varieties
\begin{equation*} 
\xymatrix@C=15pt{
	{\mathbb{F}_1 \coloneqq \left(\begin{array}{c c | c c c c c c c}
		t & s & x_1 & x_2 & x_3 & y_1 & y_2 & y_3 & y_4 \\
		0 & 5 & 1 & 1 & 4 & 9 & 8 & 7 & 6 \\
		1 & 1 & 0 & 0 & 0 & -1 & -1 & -1 & -1
		\end{array}\right) \; ,} & 
	{\mathbb{F}_1' \coloneqq\left(\begin{array}{c c | c c c c c c c}
		t & y_1 & x_1 & x_2 & y_2 & y_3 & y_4 & s & x_3 \\
		0 & 9 & 1 & 1 & 8 & 7 & 6 & 5 & 4 \\
		1 & 1 & 0 & 0 & 0 & -1 & -1 & -1 & -1
		\end{array}\right)}
} \; .
\end{equation*}
For instance, for $p_1 = -x_1^6 + x_1^5 x_2 + x_2^6 + y_4$, $p_2 = x_1^6 x_2 - x_2^7 - y_3$, $p_3 = -x_1^4 x_3 - x_3^2 - y_2$
the equations of $Y_1$ and $W_1$ are

\noindent\begin{minipage}{.5\linewidth}
	\begin{equation*}
	\begin{cases*}
	-t x_3 y_3 - s y_4 + x_1^6 x_2 x_3 - x_2^7 x_3 =0 \\
	t y_4^2 + x_1^6 y_4 - x_1^5 x_2 y_4 - x_2^6 y_4 - x_3 y_2 =0 \\
	t^2 y_4^2 + t x_1^6 y_4 - t x_1^5 x_2 y_4 - t x_2^6 y_4 + s y_3 \\ 
	+ x_1^4 x_3^2 + x_3^3 =0 \\
	t y_3 y_4 + x_1^6 y_3 - x_1^5 x_2 y_3 - x_2^6 y_3 - x_3 y_1 =0 \\
	-t^2 y_3 y_4 + t x_1^6 x_2 y_4 - t x_1^6 y_3 + t x_1^5 x_2 y_3 \\
	- t x_2^7 y_4 + t x_2^6 y_3 - s y_2 + x_1^{12} x_2 -
	x_1^{11} x_2^2 \\
	- 2 x_1^6 x_2^7 	+ x_1^5 x_2^8 + x_2^{13} =0 \\
	-t x_2 y_3 y_4 + t y_2 y_4 - t y_3^2 + x_1^5 x_2^2 y_3 +
	x_1^4 x_3 y_4 \\
	+ x_2 x_3 y_1 + x_3^2 y_4 =0 \\
	t^2 y_2 y_4 + t x_1^6 y_2 - t x_1^5 x_2 y_2 + t x_1^4 x_3 y_4 -
	t x_2^6 y_2 \\ 
	+ t x_3^2 y_4 + s y_1 + x_1^{10} x_3 -
	x_1^9 x_2 x_3 + x_1^6 x_3^2 \\
	- x_1^5 x_2 x_3^2 -
	x_1^4 x_2^6 x_3 - x_2^6 x_3^2 =0 \\
	-y_1 y_4 + y_2 y_3 =0 \\
	-t y_1 y_3 + t y_2^2 + x_1^6 x_2 y_1 + x_1^4 x_3 y_2 \\
	- x_2^7 y_1 + x_3^2 y_2 =0
	\end{cases*}
	\end{equation*}
\end{minipage}%
\begin{minipage}{.5\linewidth}
	\begin{equation*}
	\begin{cases*}
	-y_3 x_3 + y_4 s =0 \\
	t y_4^2 + x_1^6 y_4 - x_1^5 x_2 y_4 - x_2^6 y_4 - y_2 x_3 =0 \\
	t^2 x_3^3 + t x_1^4 x_3^2 - t y_3 s + x_1^6 x_2 s - x_2^7 s + 2 y_2 x_3 =0 \\
	t y_3 y_4 + x_1^6 y_3 - x_1^5 x_2 y_3 - x_2^6 y_3 - y_2 s =0 \\
	t^2 y_3 y_4 - t x_1^6 x_2 y_4 + t x_1^6 y_3 - t x_1^5 x_2 y_3 + t x_2^7 y_4 \\
	- t x_2^6 y_3 + y_1 x_3 - x_1^{12} x_2 + x_1^{11} x_2^2 
	\\+ 2 x_1^6 x_2^7 - x_1^5 x_2^8 - x_2^{13} =0 \\
	t^2 y_4 x_3^2 + t x_1^4 y_4 x_3 - t x_2 y_3 y_4 \\
	- t y_3^2 + x_1^5 x_2^2 y_3 + x_2 y_2 s + 2 y_2 y_4 =0 \\
	-t^2 x_1^6 x_3^2 + t^2 x_1^5 x_2 x_3^2 + t^2 x_2^6 x_3^2
	- t^2 y_3^2 - t x_1^{10} x_3 \\
	+ t x_1^9 x_2 x_3 + t x_1^6 x_2 y_3 
	+ t x_1^4 x_2^6 x_3 - 	t x_2^7 y_3 \\
	- y_1 s - 2 x_1^6 y_2 + 2 x_1^5 x_2 y_2 +
	2 x_2^6 y_2 =0 \\
	t y_2 y_3 + y_1 y_4 - x_1^6 x_2 y_2 + x_2^7 y_2 =0 \\
	-t^3 y_4^2 x_3 - t^2 x_1^6 y_4 x_3 + t^2 x_1^5 x_2 y_4 x_3 \\
	- t^2 x_1^4 y_4^2 +
	t^2 x_2^6 y_4 x_3 - t x_1^{10} y_4 \\
	+ t x_1^9 x_2 y_4 + t x_1^4 x_2^6 y_4 -
	y_1 y_3 - 2 y_2^2 =0
	\end{cases*}
	\end{equation*}
\end{minipage}
respectively. 
The birational links in \eqref{two links diagram} are obtained by variation of GIT quotient on $\mathbb{F}_1, \mathbb{F}_1'$ respectively, and by intersecting the contracted loci of the maps $\Psi_i, \Xi_i$ with $Y_1, W_1$ and their proper transforms (cf \cite{CampoSarkisov,BigTableLinks} for a more detailed description). 
Then, $\Psi_1$ is a flop of five rational curves and $\Psi_2$ is a hypersurface flip with weights $(9,1,1,-1,-2; 5)$; on the other hand, $\Xi_1$ flops nine rational curves and $\Xi_2$ is a hypersurface flip with weights $(1,1,8,-1,-2; 6)$. Both $\Psi_2$ and $\Xi_2$ are followed by two consecutive divisorial contractions to points. 

We carry out the explicit calculations to find the equation of $X'$; the one for $X''$ is analogous. 

The last divisorial contraction $\Phi_2'$ is a blow-down to a point in $\Proj^7(6_t,11_s,1_{x_1}, 1_{x_2}, 4_{x_3}, 3_{y_1}, 2_{y_2}, 1_{y_3})$ after setting $y_4=1$. 
Looking at the equations of $Y_4$ (that are the same as $Y_1$), we see that the coordinates $s, t, y_1$ can be globally eliminated respectively by the unprojection equations, $\Pf_5(M), \Pf_2(M)$. In particular, the (global) expression of $t, y_1$ is $t = - x_1^6 + x_1^5 x_2 + x_2^6 + x_3 y_2$, $y_1 = y_2 y_3$. 
This leaves us with three remaining equations: $ \Pf_1(M), \Pf_3(M), \Pf_4(M) =0 $ evaluated at the global expressions of $t, y_1$. 
It is easy to see that $\Pf_4(M)$ is now identically zero, and that $\Pf_1(M) = y_2 \Pf_3(M)$. Thus, the ideal $\Span[ \Pf_1(M), \Pf_3(M),\Pf_4(M)]$ is actually generated only by $\Pf_3(M)$. 
The equation of $X'$ is then $\Pf_3(M)=0$ after setting $y_4=1$ and replacing the global expressions of $t, y_1$. Hence, the degree 8 equation of $X' \subset \Proj^4(1_{x_1},1_{x_2},1_{y_3},2_{y_2},4_{x_3})$ is
\begin{equation*}
X' = \{ x_1^6 x_2 y_3 - x_2^7 y_3 - x_1^6 y_2 + x_1^5 x_2 y_2 + x_2^6 y_2 + x_3 y_2^2 + x_1^6 y_3^2 - x_1^5 x_2 y_3^2 - x_2^6 y_3^2 - x_3 y_2 y_3^2 +
x_1^4 x_3 + x_3^2 =0 \}  \; .
\end{equation*}
Analogously, the equation of $X'' \subset \Proj^4(1_{x_1},1_{x_2},1_{s},2_{y_4},4_{t})$ is
\begin{equation*}
X'' = \{ t^2 + t x_1^4 - t y_4 s^2 + x_1^6 x_2 s - x_2^7 s + 2 t y_4^2 + 2 x_1^6 y_4 - 2 x_1^5 x_2 y_4 - 2 x_2^6 y_4  =0 \}  \; .
\end{equation*}

The compound Du Val singularities of $X'$ lie at the coordinate points $P_{y_3}$ and $P_{y_2}$; the latter is also a cyclic quotient singularity, and we will only focus on the former. 
To see which kind of compound Du Val singularity $P_{y_3}$ is, we need to localise, i.e.~set $y_3=1$. We see that the lowest-power monomials are $x_3^2 + x_3 y_2 + x_1^6$, so $P_{y_3}$ is a $cA_5$ singularity. 

Instead, the compound Du Val singularities of $X''$ lie at $P_s$ and
$P_{y_4}$ (which is an orbifold point). Localising at $s=1$ we obtain the local equation $t^2 + t y_4 + x_2^7$, which is a $cA_6$ singularity. 

In conclusion we have just built two non-isomorphic $X_8$ hypersurfaces lying in the same birational class.

\subsection*{Declarations}

This work was supported by Korea Institute for Advanced Study, grant No. MG087901. The author has no relevant financial or non-financial interests to disclose.

\bibliography{bibliography}
\bibliographystyle{alpha}

\end{document}